\documentclass[a4paper,11pt,twoside,english]{article}
\usepackage[T1]{fontenc}
\usepackage[utf8]{inputenc}
\usepackage{lmodern}
\usepackage[a-1b]{pdfx}
\usepackage{hyperref}
\usepackage{cite}
\usepackage{listings}
\usepackage{color}
\usepackage{url}
\usepackage{multirow}
\usepackage{amsmath,amsthm,amssymb,amsfonts}
\usepackage{verbatim}
\usepackage{graphicx}
\usepackage{mathtools}
\usepackage{xcolor}
\usepackage{geometry}
\usepackage{enumitem}
\usepackage{ulem}
\usepackage{comment}

\newtheorem{theorem}{Theorem}[section]
\newtheorem{proposition}[theorem]{Proposition}
\newtheorem{lemma}[theorem]{Lemma}
\newtheorem{corollary}[theorem]{Corollary}

\theoremstyle{definition}

\newtheorem{remark}[theorem]{Remark}

\hypersetup{%
    pdfpagemode={UseOutlines},
    bookmarksopen,
    pdfstartview={FitH},
    colorlinks,
    linkcolor={blue},
    citecolor={blue},
    urlcolor={blue}
}

\def\e{\varepsilon}

\title{Homogenization effects on non-local functionals}
\author{
    Enrico Micalizio\thanks{\textbf{Correspondence:} Email: enrico.micalizio@studenti.polito.it} \\[1ex]
    \small Politecnico di Torino, Dipartimento di Scienze Matematiche (DISMA) \\
    \small Corso Duca degli Abruzzi 24, 10129 Torino, Italy
}
\date{}

\begin{document}
\maketitle
\begin{abstract}
  We study the homogenization of a class of non-local functionals featuring a rapidly oscillating periodic weight. By means of two-scale convergence, we explicitly evaluate the $\Gamma$-limit for constant target functions, revealing how the interplay between periodicity and non-locality forces the minimizing sequences to develop highly oscillating microstructures. As a natural consequence, we establish that the effective macroscopic functional fails to admit a standard double-integral representation. 
\end{abstract}
\vspace{0.5cm}

\section{Introduction}
The macroscopic behavior of materials and physical systems is often linked to energy minimization. 
However, as pointed out in \cite{Braides_causin, Daneri_2018, Giuliani_2006}, competing interactions at the microscopic level prevent the formation of homogeneous ground states and instead induce the formation of patterns and microstructures, leading to non-trivial and often unexpected macroscopic properties. 
This complex phenomenology naturally leads to the study of long-range interactions, whose features are effectively captured by the mathematical framework of non-local functionals. Furthermore, this class of functionals has recently attracted significant attention within the Calculus of Variations. Indeed, beyond the aforementioned long-range interactions, non-local energies -- typically formulated as double integrals -- naturally arise in models involving peridynamics \cite{Emmrich_2015, Silling2000} and image processing \cite{brezis2016, Gilboa}. In this context, recent works by Braides and Dal Maso have investigated some fundamental properties of functionals taking the form
\[
F_k(u) = \int_{\Omega \times \Omega} f_k(u(x) - u(y))\,\mathrm{d}\mu_k(x,y) + \int_\Omega g_k(x,\nabla u(x))\,\mathrm{d}x,
\]
defined on $W^{1,p}_0(\Omega)$ with $p>1$, where $\Omega$ is a bounded open set in $\mathbb{R}^d$ for $d\geq 1$.  In \cite{braidesdm2022}, the authors provide suitable conditions on the measures $\mu_k$ guaranteeing the stability of the $\Gamma$-limit, while in \cite{braides2022} it is shown that the $\Gamma$-limit retains a decoupled integral form, highlighting how the effective local part arises from the interaction between local and non-local terms. The integral representability of the limit, however, is a delicate issue: while positive results can be achieved for $p=2$ \cite{Braides23}, non-representability counterexamples emerge for $p \neq 2$ \cite{Braides23} or when the local term is dropped \cite{Braides24}.

A natural further step in this investigation is to understand how these non-local energies behave when the interaction itself oscillates at a microscopic scale $\varepsilon>0$.
When dealing with materials exhibiting a heterogeneous yet periodic microstructure, homogenization provides a powerful tool to derive the effective macroscopic behavior. We refer to \cite{bertazzoni2025} for an in-depth analysis via Young measures, and to \cite{BraidesPiatnitski, Braides_2021} for the homogenization of convolution-type energies.

Motivated by these results, in this paper we study the asymptotic behavior of the sequence of non-local functionals $F_\varepsilon \colon L^1(\Omega) \to \mathbb{R} \cup \{+\infty\}$ defined by
\begin{equation}\label{Feu}
        F_\varepsilon(u) = \int_{\Omega \times \Omega} a\left(\frac{x-y}{\varepsilon}\right) f(u(x)-u(y)) \, \mathrm{d}x \mathrm{d}y.
\end{equation}
Here, $f$ is the same triple-well potential considered in \cite{Braides24}, given by
\begin{equation}\label{potential}
    f(z) = \begin{cases} 
      0 & \text{if } z \in \{-1, 1\}, \\ 
      1 & \text{if } z = 0, \\
      +\infty & \text{otherwise,}
   \end{cases}
\end{equation}
while the periodic weight $a \in L^\infty(\mathbb{R})$ is assumed to be strictly positive. To tackle this problem, we apply the theory of two-scale convergence introduced by Nguetseng and Allaire \cite{nguetseng, allaire}, which allows us to effectively capture the microscopic oscillations during the limit process. By explicitly characterizing the $\Gamma$-limit for constant target functions, which identify the ground state energy of the system, we prove that energy minimization strictly enforces the emergence of highly oscillating step-like microstructures. 

The paper is organized as follows. In Section \ref{sec2}, after reformulating the problem using characteristic functions, we compute the two-scale limit of the functionals, highlighting the structural properties of the minimizing sequences. In Section \ref{sec3}, we solve the cell problem, showing how the minimizing profiles strongly depend on the structure of the kernel $a$. While proving that the minimum of the energy is attained by constant functions, we illustrate how interactions associated with the zero-cost wells of the potential $f$ are energetically favored. Finally, in Section \ref{sec4} we establish that the $\Gamma$-limit of the sequence of functionals defined in \eqref{Feu} does not admit a standard non-local double integral representation.

\section{Two-scale Limit and Homogenization}\label{sec2}
 Without loss of generality, we can choose $\Omega = (0,1)$ for simplicity.

\begin{lemma}\label{lemma_structure}
Let $u \in L^1(\Omega)$ and let $\{u_\varepsilon\}$ be a sequence weakly converging to $u$ in $L^1(\Omega)$ such that $\sup_\varepsilon F_\varepsilon(u_\varepsilon) < +\infty$. Then, there exists a bounded sequence $\{z_\varepsilon\} \subset \mathbb{R}$ such that, up to subsequences, $z_\varepsilon \to z \in \mathbb{R}$ and
\[
u_\varepsilon(x) \in \{z_\varepsilon, z_\varepsilon + 1\} \quad \text{for almost every } x \in \Omega.
\]
Furthermore, the limit $z$ satisfies
\begin{equation}\label{esssup}
z \leq \operatorname{ess-inf}u \quad \text{ and } \quad \operatorname{ess-sup}u \leq z+1.
\end{equation}
\end{lemma}

\begin{proof}
The strict positivity of $a$ implies that the only way to achieve $F_\varepsilon(u_\varepsilon) < +\infty$ is for $f(u_\varepsilon(x) - u_\varepsilon(y))$ to take values in $\{0,1\}$ for almost every $(x,y) \in \Omega \times \Omega$. Therefore, fixing a suitable $y \in \Omega$, there exists $z_\varepsilon$ such that $u_\varepsilon(x) \in \{z_\varepsilon, z_\varepsilon - 1, z_\varepsilon + 1\}$ for almost every $x \in \Omega$. If both values $z_\varepsilon -1$ and $z_\varepsilon +1$ were taken on sets of positive measure, there we would have $F_\varepsilon(u_\varepsilon) = +\infty$; hence we can restrict the values to $u_\varepsilon(x) \in \{z_\varepsilon, z_\varepsilon + 1\}$. Since the sequence is weakly convergent, $\{z_\varepsilon\}$ is bounded, and we can assume $z_\varepsilon \to z$ up to subsequences. Given that $z_\varepsilon \leq u_\varepsilon \leq z_\varepsilon + 1$, integrating over any measurable subset $A \subset \Omega$ and passing to the limit yields
\[
|A|z \leq \int_A u(x) \, \mathrm{d}x \leq |A|(z + 1).
\]
Since this holds for every measurable subset $A$, inequality \eqref{esssup} follows almost everywhere, which in turn implies that $u \in L^\infty(\Omega)$.
\end{proof}

Following Lemma \ref{lemma_structure}, we define the characteristic functions
\begin{equation}\label{chie}
\chi_\varepsilon(x) \coloneqq
\begin{cases}
1 & \text{if } u_\varepsilon(x) = z_\varepsilon + 1 \\
0 & \text{if } u_\varepsilon(x) = z_\varepsilon,
\end{cases}
\end{equation}
so that 
\[
f(u_\varepsilon(x) - u_\varepsilon(y)) = \chi_\varepsilon(x)\chi_\varepsilon(y) + (1-\chi_\varepsilon(x))(1-\chi_\varepsilon(y))
\]
and the functional rewrites as
\begin{equation}\label{Feps}
F_\varepsilon(u_\varepsilon) = \int_{\Omega \times \Omega }a\left(\frac{x-y}{\varepsilon}\right) \Big[\chi_\varepsilon(x)\chi_\varepsilon(y) + (1-\chi_\varepsilon(x))(1-\chi_\varepsilon(y))\Big] \, \mathrm{d}x \mathrm{d}y.
\end{equation}
We pass to the limit by using a two-scale convergence approach. 

Let $Y \coloneqq (0,1)$ be the unit cell. Note that since $\chi_\varepsilon \in \{0,1\}$ almost everywhere and~$\Omega$ has finite measure, the sequence $\{\chi_\varepsilon\}$ is uniformly bounded not only in $L^\infty(\Omega)$ but also in $L^2(\Omega)$. This allows us to apply classical two-scale convergence theory. By \cite[Theorem~1.2]{allaire}, we can extract a subsequence that two-scale converges to a limit in $L^2(\Omega \times Y)$, namely there exists a function $\phi\in L^2(\Omega \times Y)$ such that $\chi_\varepsilon \overset{2}{\rightharpoonup} \phi$.

\begin{lemma}\label{lemmawe}
If $\chi_\varepsilon \overset{2}{\rightharpoonup}\phi$,  
then 
$w_\varepsilon(x,y) \coloneqq \chi_\varepsilon(x)\chi_\varepsilon(y) \overset{2}{\rightharpoonup} \phi(x,\sigma)\phi(y,\tau)$.
\end{lemma}
\begin{proof}
We test the bounded sequence $\{w_\varepsilon\}\subset L^2(\Omega \times \Omega)$ against functions in $\mathcal{D}(\Omega \times \Omega; C^\infty_\#(Y\times Y))$, which is the space of smooth, compactly supported functions on $\Omega \times \Omega$, taking values in the space of smooth, $(Y \times Y)$-periodic functions. By invoking the Stone--Weierstrass Theorem, we may restrict our analysis to 
test functions of the form $\Psi(x,y,\sigma,\tau) = \psi_1(x,\sigma)\psi_2(y,\tau)$, with $\psi_1, \psi_2 \in \mathcal{D}(\Omega; C^\infty_\#(Y))$. We obtain
\begin{equation*}
\begin{split}
&\lim_{\varepsilon \to 0}\int_{\Omega \times \Omega} w_\varepsilon(x,y)\Psi \biggl(x,y,\frac{x}{\varepsilon},\frac{y}{\varepsilon}\biggr) \, \mathrm{d}x \mathrm{d}y \\
= &\lim_{\varepsilon \to 0}\left(\int_\Omega\chi_\varepsilon(x)\psi_1\left(x,\frac{x}{\varepsilon}\right)\mathrm{d}x \right)\left(\int_\Omega\chi_\varepsilon(y)\psi_2\left(y,\frac{y}{\varepsilon}\right)\mathrm{d}y \right) \\
= &\int_{\Omega \times \Omega}\int_Y\int_Y\phi(x,\sigma)\phi(y,\tau)\psi_1(x,\sigma)\psi_2(y,\tau) \, \mathrm{d}\sigma  \mathrm{d}\tau \mathrm{d}x \mathrm{d}y,
\end{split}
\end{equation*}
where we used the Fubini--Tonelli Theorem; the thesis follows by density.
\end{proof}


We define the functional $F\colon L^\infty(Y;[0,1])\to\mathbb{R}$ by
\begin{equation}\label{Fvarphi}
 F(\varphi) \coloneqq \int_Y \int_Y a(\sigma-\tau) \Big[\varphi(\sigma)\varphi(\tau) + (1-\varphi(\sigma))(1-\varphi(\tau))\Big] \, \mathrm{d}\sigma \mathrm{d}\tau.
\end{equation}
\begin{proposition}\label{hom}
Let $u \in L^1(\Omega)$ and let $\{u_\varepsilon\}$ be a sequence weakly converging to $u$ in~$L^1(\Omega)$ such that $\sup_\varepsilon F_\varepsilon(u_\varepsilon) < +\infty$.
Let $\chi_\varepsilon$ be the characteristic functions associated with $u_\varepsilon$ defined by \eqref{chie}, and suppose they admit a two-scale limit $\phi\in L^2(\Omega \times Y)$.
Upon defining $\varphi(\cdot) \coloneqq \int_\Omega\phi(\xi,\cdot)\,\mathrm{d}\xi$,
we have
\begin{equation}
    \lim_{\varepsilon \to 0}F_\varepsilon(u_\varepsilon) = F(\varphi).
\end{equation}
\end{proposition}
\begin{proof}
While the two-scale convergence theorem \cite[Theorem 1.2]{allaire} typically requires test functions to be in $ L^2(\Omega; C_{\#}(Y))$, a standard density argument of $C^0$ in $L^1$ allows us to use $a \in L^\infty_{\#}(Y)$ as an admissible test function. 
Lemma~\ref{lemmawe} along with the Fubini--Tonelli Theorem yields
\begin{equation*}
\begin{split}
\lim_{\varepsilon \to 0}\int_{\Omega \times \Omega} a\left(\frac{x-y}{\varepsilon}\right) \chi_\varepsilon(x)\chi_\varepsilon(y)\,\mathrm{d}x\mathrm{d}y &= \int_{\Omega \times \Omega} \int_Y \int_Y a(\sigma-\tau)\phi(x,\sigma)\phi(y,\tau) \, \mathrm{d}\sigma  \mathrm{d}\tau  \mathrm{d}x  \mathrm{d}y \\
&= \int_Y \int_Y a(\sigma-\tau)\varphi(\sigma) \varphi(\tau) \, \mathrm{d}\sigma  \mathrm{d}\tau.
\end{split}
\end{equation*}
Defining $v_\varepsilon(x,y) \coloneqq (1-\chi_\varepsilon(x))(1-\chi_\varepsilon(y))$, Lemma \ref{lemmawe} similarly ensures $v_\varepsilon(x,y)  \overset{2}{\rightharpoonup} (1-\phi(x,\sigma))(1-\phi(y,\tau))$, and the conclusion follows by linearity.
\end{proof}

\begin{corollary}
Let $u \in L^1(\Omega)$ and let $\{u_\varepsilon\}$ be a sequence weakly converging to $u$ in $L^1(\Omega)$ such that $\sup_\varepsilon F_\varepsilon(u_\varepsilon) < +\infty$, and let $z \in \mathbb{R}$ be the limit of the associated sequence~$\{z_\varepsilon\}$ given by Lemma \ref{lemma_structure}.
By letting $t \coloneqq \int_Y \varphi(\sigma)\,\mathrm{d} \sigma$, the limit function $u$ satisfies the relation
\begin{equation}\label{la8}
\int_\Omega u(x) \, \mathrm{d}x = z + t.
\end{equation}
\end{corollary}
\begin{proof}
From \cite[Proposition 1.6]{allaire}, the two-scale convergence $\chi_\varepsilon \overset{2}{\rightharpoonup} \phi$ implies the weak convergence $\chi_\varepsilon(\cdot) \rightharpoonup \int_Y\phi(\cdot,\sigma)\,\mathrm{d}\sigma$ in $L^2(\Omega)$, and thus in $L^1(\Omega)$. 
By integrating over~$\Omega$, we get 
\[
\int_\Omega \chi_\varepsilon(x)\,\mathrm{d}x \to \int_\Omega\left(\int_Y\phi(x,\sigma) \, \mathrm{d}\sigma\right) \mathrm{d}x = \int_Y \varphi(\sigma)\,\mathrm{d}\sigma = t.
\]
By Lemma \ref{lemma_structure}, we can write $u_\varepsilon(x) = z_\varepsilon + \chi_\varepsilon(x)$. Integrating and passing to the limit yields the conclusion.
\end{proof}

\section{Minimization of the Cell Problem}\label{sec3}
To explicitly investigate the optimal microstructures governing the homogenized functional, we specialize our analysis to a symmetric, non-trivial piece-wise constant periodic weight.
For fixed parameters $\beta > \alpha > 0$ and a given $\lambda \in (0,1)$, we consider a symmetric profile for the kernel $a_\lambda$, centered around $t = 1/2$:
\[
a_\lambda(t) = \begin{cases}\alpha & \text{if } t \in [0, \frac{\lambda}{2}) \cup [1-\frac{\lambda}{2}, 1)\\ \beta & \text{if } t \in [\frac{\lambda}{2}, 1-\frac{\lambda}{2}),\end{cases}
\]
extended periodically to $\mathbb{R}$. 
Without loss of generality, we can restrict ourselves to the case $0 < \lambda \le 1/2$.
The average of the weight is denoted by $\overline{a}_\lambda = \lambda \alpha + (1-\lambda)\beta$.

The sequence of non-local functionals associated with this weight is
\begin{equation}\label{Feul}
        F_{\varepsilon}^{\lambda}(u) = \int_{\Omega \times \Omega} a_\lambda \left(\frac{x-y}{\varepsilon}\right) f(u(x)-u(y)) \, \mathrm{d}x \mathrm{d}y.
\end{equation}
As proved in Proposition \ref{hom}, we have that
\begin{equation}\label{Fvarphi_l}
\lim_{\varepsilon \to 0}F_\varepsilon^\lambda(u_\varepsilon) = \int_Y \int_Y a_\lambda(\sigma-\tau) \Big[\varphi(\sigma)\varphi(\tau) + (1-\varphi(\sigma))(1-\varphi(\tau))\Big] \, \mathrm{d}\sigma \mathrm{d}\tau \eqqcolon F^\lambda(\varphi),
\end{equation}
which we aim to minimize over all admissible microscopic profiles $\varphi$ subject to the volume fraction constraint $\int_Y \varphi(\sigma)\,\mathrm{d}\sigma\eqqcolon t\in[0,1]$. Thus, we define the cell problem energy:
\begin{equation}\label{gamma_t}
\gamma_\lambda(t) \coloneqq \min \left\{ F^\lambda(\varphi) : \varphi \in L^\infty(Y; [0,1])  \text{ with } \int_Y \varphi(\sigma)\,\mathrm{d}\sigma = t \right\}.
\end{equation}

\begin{proposition}\label{lemma3.3}
    The optimal profile minimizing the cell problem \eqref{gamma_t} is given by the characteristic function $\varphi_t(x) = \chi_{[0, t/2) \cup (1-t/2, 1)}(x)$ modulo 1, and the explicit minimum energy is
\begin{equation}\label{gamma_explicit}
\gamma_\lambda(t) = \begin{cases}
2\alpha t^2 - 2\overline{a}_\lambda t + \overline{a}_\lambda & \text{if $t \in[0, \frac{\lambda}{2}]$,} \\[1ex]
2\beta(t^2-t) - \frac{\alpha -\beta}{2}\lambda^2 + \overline{a}_\lambda & \text{if $t\in [\frac{\lambda}{2}, 1-\frac{\lambda}{2}],$} \\[1ex]
2\alpha(1-t)^2 + 2\overline{a}_\lambda t - \overline{a}_\lambda & \text{if $t\in [1-\frac{\lambda}{2}, 1]$.}
\end{cases}
\end{equation}
\end{proposition}
\begin{proof}
Let us expand the integrand defining $F^\lambda(\varphi)$:
\[
a_\lambda(\sigma - \tau)\big[\varphi(\sigma)\varphi(\tau) + (1-\varphi(\sigma))(1-\varphi(\tau))\big] = a_\lambda(\sigma - \tau)\big[2\varphi(\sigma)\varphi(\tau) - \varphi(\sigma) - \varphi(\tau) + 1\big].
\]
Integrating this over $Y \times Y$ and noting that $\int_Y a_\lambda(\sigma-\tau)\mathrm{d}\tau = \overline{a}_\lambda$ for any fixed $\sigma$, the functional simplifies to
\[
F^\lambda(\varphi) = 2J^\lambda(\varphi) - 2\overline{a}_\lambda t + \overline{a}_\lambda,
\]
where $J^\lambda(\varphi) \coloneqq \int_Y \int_Y a_\lambda(\sigma-\tau)\varphi(\sigma)\varphi(\tau) \, \mathrm{d}\sigma \mathrm{d}\tau$.

Since $t$ and $\overline{a}_\lambda$ are fixed, minimizing~$F^\lambda$ is equivalent to minimizing~$J^\lambda$\,. We can rewrite the weight as $a_\lambda(x) = \alpha + (\beta - \alpha)\chi_{(\lambda/2, 1-\lambda/2)}(x)$. 
Consequently, the functional~$J^\lambda$ reads
\begin{align*}
    J^\lambda(\varphi) &= \int_Y \int_Y\big[\alpha + (\beta - \alpha)\chi_{(\lambda/2, 1-\lambda/2)}(\sigma-\tau)\big] \varphi(\sigma)\varphi(\tau)\,\mathrm{d}\sigma \mathrm{d}\tau\\[1ex]
    &= \alpha t^2 + (\beta - \alpha) \int_Y\int_Y \chi_{(\lambda/2, 1-\lambda/2)}(\sigma-\tau)\varphi(\sigma)\varphi(\tau)\,\mathrm{d}\sigma \mathrm{d}\tau \\[1ex]
    &= \alpha t^2 + (\beta - \alpha)\left[t^2 - \int_Y \int_Y \chi_{(0, \frac{\lambda}{2}) \cup (1-\frac{\lambda}{2}, 1)}(\sigma-\tau)\varphi(\sigma)\varphi(\tau) \,\mathrm{d}\sigma \mathrm{d}\tau  \right],
\end{align*}
so that minimizing it is, in turn, equivalent to maximizing (recall that $\beta-\alpha>0$) the term
\[
K^\lambda(\varphi) \coloneqq \int_Y \int_Y \chi_{(-\lambda/2,\lambda/2)}(\sigma-\tau)\varphi(\sigma)\varphi(\tau) \, \mathrm{d}\sigma \mathrm{d}\tau,
\]
where the intervals are understood modulo 1. Since the kernel $\chi_{(-\lambda/2,\lambda/2)}$ is symmetrically decreasing for any $\lambda \in (0, 1/2]$, the Riesz--Sobolev rearrangement inequality implies that $K^\lambda$ is maximized when $\varphi$ is also a symmetrically decreasing indicator function. Hence, the optimal profile is 
\begin{equation*}
    \varphi_t(x) \coloneqq \chi_{(-t/2, t/2)}(x) \equiv \chi_{[0, t/2) \cup (1-t/2, 1)}(x)\quad \text{modulo 1},
\end{equation*}
which satisfies the constraint $\int_Y \varphi(\sigma)\,\mathrm{d}\sigma = t$.

Computing the integral $J^\lambda(\varphi_t) = \int_{-t/2}^{t/2} \int_{-t/2}^{t/2} a_\lambda(\sigma-\tau) \, \mathrm{d}\sigma \mathrm{d}\tau$ for the three regimes $t < \lambda/2$, $t \in [\lambda/2, 1-\lambda/2)$, and $t \ge 1-\lambda/2$ yields the polynomial expressions in~\eqref{gamma_explicit}, whose continuity with respect to~$t$ is immediate to verify.
\end{proof}
\begin{remark}
It is worth noting that the assumption $\alpha < \beta$ drives the symmetric rearrangement to concentrate the mass of the optimal profile $\varphi_t$ around the boundary of the unit cell (i.e., around $t=0$), precisely where the low-cost kernel $\alpha$ is distributed.

If, conversely, one considers $\alpha > \beta$, the energetic penalty is inverted. To minimize the intersection with the expensive $\alpha$-phase, the optimal profile concentrates its mass symmetrically within the central low-cost plateau, yielding
    \[
        \varphi_t(x) = \chi_{(1/2-t/2, 1/2+t/2)}(x).
    \]
However, the minimum energy $\gamma_\lambda(t)$ has exactly the same shape as in \eqref{gamma_explicit}, except that~$\alpha$ and $\beta$ are swapped, as well as $\lambda$ and $1-\lambda$.
In the sequel, it may be useful to take this into consideration.
\end{remark}
We define the final homogenized functional evaluated on a generic function $u \in L^1(\Omega)$ as:
 \begin{equation}\label{Fbar}
 \overline{F}^\lambda(u) \coloneqq \min \bigl\{ \gamma_\lambda(t) : t\in I_u\coloneqq [\iota(u), \varsigma(u)] \bigr\},
 \end{equation}
where $\iota(u) \coloneqq \int_\Omega u(x) \, \mathrm{d}x - \operatorname{ess-inf} u$ and $\varsigma(u) \coloneqq \int_\Omega u(x) \, \mathrm{d}x - \operatorname{ess-sup} u + 1$. Note that if the essential oscillation of $u$ is strictly greater than 1 (i.e., $\operatorname{ess-sup} u - \operatorname{ess-inf} u > 1$), the interval $I_u$ is empty  and $\overline{F}^\lambda(u) =\min_{t \in \emptyset}\{\gamma_\lambda(t)\}= +\infty$, which is consistent with the conditions established by Lemma \ref{lemma_structure}.

 
\begin{theorem}\label{thm3.2}
Let $u \equiv c \in \mathbb{R}$ be a constant function. Then, the $\Gamma$-limit of the sequence $\{F_\varepsilon^\lambda\}$ exists and evaluates exactly to the minimum energy of the cell problem:
\begin{equation}\label{lim1}
    \Big(\Gamma-\lim_{\varepsilon \to 0} F^\lambda_\varepsilon\Big)(c) = \overline{F}^\lambda(c) = \frac{\bigl(1 - (1-\lambda)^2\bigr)\alpha + (1-\lambda)^2\beta}{2}.
\end{equation}
\end{theorem}

    

\begin{proof}
\smallskip\noindent\textit{Liminf inequality.} 
From Proposition \ref{hom} and formulas \eqref{gamma_t} and \eqref{Fbar}, we deduce that for any $u_\e \rightharpoonup c$ in $L^1(\Omega)$ with bounded energy, it holds
\[
    \liminf_{\e \to 0}F^\lambda_\e(u_\e) = F^\lambda(\varphi) \geq \gamma_\lambda(t) \geq \overline{F}^\lambda(c).
\]

\smallskip\noindent\textit{Limsup inequality.} 
For the constant function $u \equiv c$, we have that $I_c = [0,1]$. 
The optimal parameter $\overline{t}$ that minimizes $\gamma_\lambda(t)$ over $I_c$ imposes the optimal shift to be $\overline{z} = \int_\Omega u(x) \,\mathrm{d}x - \overline{t}$, see~\eqref{la8}. 
Minimizing $\gamma_\lambda(t)$ in~\eqref{gamma_explicit} over $[0,1]$ yields that the minimizer is $\overline{t} = 1/2$; consequently the optimal shift is $\overline{z} = c - 1/2$.

Recalling the definition of $\varphi_t$ as the minimizer of the problem defining $\gamma_\lambda(t)$ (see~\eqref{gamma_t} and Proposition~\ref{lemma3.3}), we construct the recovery sequence as
\[
    u_\varepsilon^*(x) \coloneqq \overline{z} + \varphi_{\overline{t}}\left(\frac{x}{\varepsilon}\right) = c - \frac{1}{2} + \varphi_{1/2}\left(\frac{x}{\varepsilon}\right),
\]
which weakly converges in $L^1(\Omega)$ to $c - 1/2 + \int_0^1\varphi_{1/2}(\sigma) \,\mathrm{d}\sigma = c$. Plugging~$u_\varepsilon^*$ into~$F_\varepsilon$ and passing to the limit as $\varepsilon\to0$ yields precisely the integral defining $F^\lambda(\varphi_{1/2})$. Since~$\varphi_{1/2}$ is the optimal profile for $t=1/2$, this integral evaluates exactly to the cell-problem minimum $\gamma_\lambda(1/2)$. Furthermore, since $t=1/2$ is the global minimizer of $\gamma_\lambda(t)$ over the interval $[0,1]$, we obtain that $F^\lambda(\varphi_{1/2})=\overline{F}^\lambda(c)$. 
Thus, we have proved that 
\begin{equation*}
    \begin{split}
    \lim_{\varepsilon \to 0} F^\lambda_\varepsilon(u_\varepsilon^*) = \overline{F}^\lambda(c)= F^\lambda(\varphi_{1/2})=\gamma_\lambda(1/2) = \frac{\bigl(1 - (1-\lambda)^2\bigr)\alpha + (1-\lambda)^2\beta}{2}.
    \end{split}
\end{equation*}
\end{proof}

\begin{remark}
    It is worth noting how the macroscopic energy forces the interactions into the zero-cost wells of the potential $f$, rejecting a flat microscopic profile (which would cost the full average weight $\overline{a}_\lambda$, strictly greater than the expression in \eqref{lim1}) in favor of a highly oscillating recovery sequence to achieve the constant macroscopic state $u\equiv c$.  Therefore, taking the limit to the homogeneous cases, if $\lambda \to 1$ ($a(t)\equiv\alpha$), the high oscillations of the optimal profile between $c-\frac{1}{2}$ and $c + \frac{1}{2}$ halve the macroscopic energy: $\overline{F}^1(c) = \frac{\alpha}{2}$. Similarly, $\overline{F}^0(c) = \frac{\beta}{2}$ for $\lambda \to 0$. 
\end{remark}

\section{Consequences on Integral Representation}\label{sec4}
The explicit evaluation of the $\Gamma$-limit on optimal microstructures allows us to deduce a strong property of the homogenized macroscopic energy: the loss of the classical double-integral form.
To prove this, the following lemma regarding the strong convergence of characteristic functions will be useful.
\begin{lemma}\label{strong}
    Let $\{A_n\}$ be a sequence of measurable subsets of $\Omega$ and let $A \subset \Omega$. If the sequence of characteristic functions $\{\chi_{A_n}\}$ weakly converges to the characteristic function $\chi_A$ in $L^1(\Omega)$ as $n \to \infty$, then the convergence is indeed strong in $L^1(\Omega)$.
\end{lemma}
\begin{proof}
The weak convergence in $L^1$ implies that for any test function $v \in L^\infty(\Omega)$, we have $\int_\Omega \chi_{A_n} v \to \int_\Omega \chi_A v$.
Choosing $v \equiv 1$, we get $\int_\Omega \chi_{A_n} \to \int_\Omega \chi_A$, which means $|A_n| \to~|A|$.
Choosing $v = \chi_A$, we get $\int_\Omega \chi_{A_n}\chi_A \to \int_\Omega \chi_A \chi_A = \int_\Omega \chi_A$, which means that $|A_n \cap~A| \to~|A|$.
We can now conclude by expanding the $L^1$ norm of the difference:
\[
\lim_{n \to \infty} \|\chi_{A_n} - \chi_A \|_{L^1(\Omega)} =  \lim_{n \to \infty} \left( |A_n| + |A| - 2|A_n \cap A| \right) = |A| + |A| - 2|A| = 0.\qedhere
\]
\end{proof}
\begin{theorem}\label{main}
    There exists no function $g$ such that the $\Gamma$-limit of $\{F^\lambda_\varepsilon\}$ can be represented as an integral functional of the form
    \begin{equation}\label{g}
        \Big(\Gamma-\lim_{\e \to 0}F^\lambda_\e\Big)(u) = \int_{\Omega \times \Omega}g(u(x)-u(y)) \,\mathrm{d}x\mathrm{d}y,
    \end{equation}
    for every $u \in L^1(\Omega)$.
\end{theorem}
\begin{proof}
    We argue by contradiction, similarly to \cite{Braides24}, and suppose that such a representation exists. 
    We already explicitly computed the exact $\Gamma$-limit for constant functions in Theorem \ref{thm3.2}. Now, let $s \in (0,1)$ and define the target step function $u_s$ as
\[
    u_s(x) \coloneqq \begin{cases}
    1 & \text{if $x\in(0, s]$,} \\
    0 & \text{if $x \in (s,1)$.}
    \end{cases}
\]
We claim that for this specific profile, the $\Gamma$-limit evaluates to
\begin{equation}\label{lim2}
    \Big(\Gamma - \lim_{\e \to 0} F^\lambda_\e\Big)(u_s) = \overline{a}_\lambda \int_{\Omega \times \Omega} \Big[u_s(x) u_s(y) + (1-u_s(x))(1-u_s(y))\Big] \,\mathrm{d}x\mathrm{d}y \eqqcolon T^\lambda(u_s). 
\end{equation}
The value $T^\lambda(u_s)$ can be computed explicitly, noting that the integrand is $1$ when $x$ and~$y$ are on the same side of the jump $s$, and $0$ otherwise, yielding $T^\lambda(u_s) = \overline{a}_\lambda\bigl(s^2+(1-s)^2\bigr)$.

\smallskip\noindent\textit{Liminf inequality.} 
Let $\{u_\varepsilon\}$ be an arbitrary sequence weakly converging to $u_s$ in $L^1(\Omega)$. 
As a consequence of Lemma~\ref{lemma_structure}, see~\eqref{chie}, we can write $u_\varepsilon = z_\varepsilon + \chi_\varepsilon$\,. Since $u_s \in \{0,1\}$ is itself a characteristic function, the weak convergence forces $z_\varepsilon \to 0$ and $\chi_\varepsilon \rightharpoonup u_s$ in~$L^1(\Omega)$. 
By Lemma \ref{strong}, the convergence of $\{\chi_\varepsilon\}$ to $u_s$ is indeed strong in $L^1(\Omega)$. 
This allows us to pass to the limit in the term $F^\lambda_\varepsilon(u_\varepsilon)$, decoupling the strong convergence of $\chi_\varepsilon(x)\chi_\varepsilon(y)$ from the weak convergence of $a_\lambda \left(\frac{x-y}{\varepsilon}\right) \rightharpoonup \overline{a}_\lambda$, achieving the expression in~\eqref{lim2}. 

\smallskip\noindent\textit{Limsup inequality.} 
We choose the constant recovery sequence $u_\varepsilon^* \equiv u_s$. Applying the Riemann--Lebesgue Lemma to the oscillating weight $a_\lambda\left(\frac{x-y}{\varepsilon}\right)$, we immediately recover the explicit value of $T^\lambda(u_s)$.

We now match these two exact evaluations with the hypothetical formula \eqref{g}, yielding a counterexample analogous to \cite{Braides24}. For the constant function $u \equiv c$, since $u(x) - u(y) = 0$, formula \eqref{g} combined with Theorem \ref{thm3.2} implies
\[
       g(0) = \int_{\Omega \times \Omega} g(0) \,\mathrm{d}x\mathrm{d}y = \overline{F}^\lambda(c) = \frac{\bigl(1 - (1-\lambda)^2\bigr)\alpha + (1-\lambda)^2\beta}{2}.
\]
Now, evaluating \eqref{g} on the step function $u_s$ and equating it to \eqref{lim2}, we obtain
\[
    T^\lambda(u_s) = \int_{\Omega \times \Omega} g(u_s(x)-u_s(y)) \,\mathrm{d}x\mathrm{d}y = \bigl(s^2 + (1-s)^2\bigr)g(0) + 2s(1-s)g(1).
\]
By substituting the known values for $T^\lambda(u_s)$ and $g(0)$, we obtain
\[
    \overline{a}_\lambda\bigl(s^2 + (1-s)^2\bigr) = \bigl(s^2 + (1-s)^2\bigr)\frac{\bigl(1 - (1-\lambda)^2\bigr)\alpha + (1-\lambda)^2\beta}{2} + 2s(1-s)g(1),
\]
and solving for $g(1)$ gives
\[
    g(1) = \frac{\left(s^2 + (1-s)^2\right)}{2s(1-s)}\frac{\lambda^2\alpha + (1-\lambda^2)\beta}{2}.
\]
Consequently, the value of $g(1)$ depends explicitly on the arbitrary jump point $s \in (0,1)$, which is a contradiction.
\end{proof}

\begin{remark}
     Our main results rely on a function $f$ that takes the value $+\infty$ almost everywhere. However, as pointed out in \cite{Braides24}, an extension to finite functions occurs by considering, for a fixed constant $M > 0$, the everywhere finite integrand
    \[
    f_M(z) = \begin{cases}
    0 & \text{if } z \in \{-1, 1\} \\
    1 & \text{if } z = 0 \\
    M & \text{otherwise,}
    \end{cases}
    \]
    which is such that $f_M \nearrow f$. By choosing a sufficiently large $M$, one can easily show that any deviation from increments in $\{-1, 0, 1\}$ becomes strictly sub-optimal. Consequently, the optimal profiles for $f_M$ rigidly coincide with those of the infinite function, yielding the same macroscopic energy on our target functions.
\end{remark}

\section*{Acknowledgments}
The author thanks Andrea Braides, Valeria Chiadò Piat and  Marco Morandotti for useful discussions.
    
\bibliographystyle{plain}
\bibliography{biblio}
\end{document}